\documentclass{article}
\usepackage{amsmath, amsthm, amsfonts, amssymb}
\usepackage{hyperref}
\newtheorem{theorem}{Theorem}[section]
\newtheorem{lemma}[theorem]{Lemma}
\newtheorem{proposition}[theorem]{Proposition}
\newtheorem{corollary}[theorem]{Corollary}
\theoremstyle{definition}
\newtheorem{definition}[theorem]{Definition}
\newtheorem{example}[theorem]{Example}

\theoremstyle{remark}
\newtheorem{remark}[theorem]{Remark}
\numberwithin{equation}{section}

\begin{document}

\title{\bf Bicomplex Linear Operators on Bicomplex Hilbert Spaces
 and
  Littlewood's Subordination Theorem}
\date{\textbf{Romesh Kumar and Kulbir Singh}}
\vspace{0in}
\maketitle
 
$$\textbf{Abstract}$$
In this paper we study some basic properties of bicomplex linear operators on bicomplex Hilbert spaces. Further we discuss some applications of Hahn-Banach theorem on bicomplex Banach modules. We also introduce and discuss some bicomplex holomorphic function spaces and prove Littlewood's Subordination principle for bicomplex Hardy space.\\ 

\begin{section} {Introduction and Preliminaries}
\renewcommand{\thefootnote}{\fnsymbol{footnote}}

In this section we summarize the basic properties of bicomplex numbers, bicomplex linear operators and bicomplex holomorphic functions.
Let $i$ and $j$ be two commutating imaginary units, i.e.,
$$ ij=ji,  i^2=j^2=-1.$$
By $\mathbb{C}$(i), we denote the field of complex numbers of the form $x+iy$. The set of bicomplex numbers $\mathbb{B}\mathbb{C}$ is then defined as 
$$ \mathbb{B}\mathbb{C}=\left\{Z=x_0+ix_1+jx_2+ijx_3,\text{with}\; x_0,x_1,x_2,x_3\in \mathbb R \right\}$$
~~~~~~~~~~~~~~~~~~$=\left\{Z=z+jw,\text{with}\;z,w\in \mathbb{C}(i)\right\}.$\\
The set $\mathbb{B}\mathbb{C}$ turns out to be a ring with respect to the sum and product  defined by\\
 $$Z+U=(z+jw)+(u+jv)=(z+u)+j(w+v)\;,$$ $$ZU=(z+jw)(u+jv)=(zu-wv)+j(wu+zv)$$ and thus it turns out to be a module over itself. Inparticular, $\mathbb{B}\mathbb{C}$ can be seen as a vector space over $\mathbb{C}$(i). For further details, we refer the reader to \cite{YY}, \cite{MM}, \cite{KK}, \cite{RR}.\\  
 If we put $z=x$ and $w=iy$ with $x,y\in \mathbb{R}$, then we obtain the set of hyperbolic numbers 
 $$\mathbb{D}=\left\{x+yk:k^2=1 \;\text{and}\;x,y\in \mathbb{R}\; \text{with}\;k \notin \mathbb R\right\}$$
 The algebra $\mathbb{B}\mathbb{C}$ is not a division algebra, since one can see that if $e_1=\frac{1+ij}{2}$ and $e_2=\frac{1-ij}{2}$, then $e_1.\;e_2=0$.\\
 The bicomplex numbers $e_1,e_2$ are linearly independent, mutually annihilating, idempotent and satisfy $e_1+e_2=1$. Infact, $e_1$ and $e_2$ are hyperbolic numbers. They make up so called idempotent basis of bicomplex numbers.
 Any bicomplex number $Z=z+jw$ can be written uniquely as  
$$Z=z_1e_1+z_2e_2$$
and is called the idempotent representation of a bicomplex number,\\ where $z_1=z-iw$ and $z_2=z+iw$ are elements of  $\mathbb{C}$(i). By using the idempotent representation of bicomplex numbers, we can write $$\mathbb{B}\mathbb{C}=e_1X_1+e_2X_2,$$ where $X_1=\left\{z-iw\;|\;z,w\in \mathbb{C}(i)\right\}$,\;$X_2=\left\{z+iw\;|\;z,w\in \mathbb{C}(i)\right\}.$
 Further, as in \cite [ P. 40]{KK} for any open set $U$ in  $\mathbb{B}\mathbb{C}$ there exists two open sets $U_1\subseteq X_1$ and $U_2\subseteq X_2$ such that 
\begin{align}
 U=e_1U_1+e_2U_2.
\end{align}
Let $\mathbb {B}\mathbb {C}^{m\times n}$ be the set of all ${m\times n}$ matrices with bicomplex entries. For any matrix $A=[a_{ij}]\in \mathbb {B}\mathbb {C}^{m\times n}$, by using the idempotent representation of bicomplex number $a_{ij}\; \forall\; i,j$ we can write, 
$$A=B+jC=e_1A_1+e_2A_2$$ where $A_1,A_2$ are two ${m\times n}$ complex matrices see, \cite  [P. 566] {G_1G_1}

Since bicomplex numbers are defined as pair of two complex numbers connected through another imaginary unit, there are several natural notions of conjugation.\\
let $Z=z+jw\in \mathbb{B}\mathbb{C}$. We define the following three conjugates in $\mathbb{B}\mathbb{C}$:\\
(i)$Z^{\dagger_1}=\overline{z}+j\overline{w}$\;,\;\; (ii)\;$Z^{\dagger_2}=z-jw$\;,\;\;(iii)\;$Z^{\dagger_3}= \overline{z}-j\overline{w}$

With each kind of conjugation, one can define a specific bicomplex modulus  as
$$|Z|^2_j=Z\;.\;Z^{\dagger_1}$$
$$|Z|^2_i=Z\;.\;Z^{\dagger_2}$$
$$|Z|^2_k=Z\;.\;Z^{\dagger_3}$$
Since none of the moduli above is real valued, we can consider also the Euclidean norm on $\mathbb{B}\mathbb{C}$, that is, for any $Z=x_0+ix_1+jx_2+ijx_3=z+jw\in \mathbb{B}\mathbb{C}$,\\
define $|Z|=\sqrt{x_0^2+x_1^2+x_2^2+x_3^2}=\sqrt{|z|^2+|w|^2}$\;, \\
then such a norm doesnot respect the multiplicative structure of $\mathbb{B}\mathbb{C}$ see, \cite [P. 7] {KK}, since if $Z,W\in \mathbb{B}\mathbb{C}$, we have   
$$|ZW|\leq \sqrt{2}|Z|\;|W|.$$
Further the hyperbolic modulus $|Z|^2_k$ of any $Z\in \mathbb{B}\mathbb{C}$ is given by the formula $$|Z|^2_k=Z\;.\;Z^{\dagger_3}.$$
Thus writing $Z=z_1e_1+z_2e_2$, one has $|Z|_k=|z_1|e_1+|z_2|e_2$ and is called the hyperbolic norm on $\mathbb{B}\mathbb{C}$. The hyperbolic norm and euclidean norm in $\mathbb{B}\mathbb{C}$ has been intensively discussed in \cite [Section 1.3] {YY}.\\
\begin{definition} Let $X$ be a $\mathbb{B}\mathbb{C}$-module. Then (see, \cite{XX}), we can write
$$X=e_1X_1+e_2X_2,$$ where $X_1=e_1X$ and $X_2=e_2X$ are two $\mathbb{C}$(i) vector spaces so that any $x$ in $X$ can be witten as $x=e_1x+e_2x=e_1x_1+e_2x_2.$ Assume that $X_1,X_2$ are normed spaces with respective norms $\|.\|_1\;\text{and}\; \|.\|_2.$
Set $$ \|x\|=\sqrt{\frac{\|x_1\|^2_1+\|x_2\|^2_2}{2}}$$
Then $\|.\|$ defines a real valued norm on $X$ such that for any $\lambda \in \mathbb{B}\mathbb{C}$,\; $x\in X$, 
$$\|\lambda x\|\leq \sqrt{2}|\lambda|\;\|x\|$$
The $\mathbb{B}\mathbb{C}$-module $X$ can also be endowed canonically with the hyperbolic norm given by the formula 
$$\|x\|_{\mathbb D}=\|e_1x_1+e_2x_2\|_{\mathbb D}=\|x\|_1e_1+\|x_2\|_2e_2$$ such that for any $\lambda\in \mathbb{B}\mathbb{C},\;x\in X$, we have $||\lambda x||_{\mathbb D}=|\lambda|_k\;||x||_{\mathbb D}$. 
Note that these norms are connected by the equality \begin{align} |\|x\|_{\mathbb D}|=\|x\| 
\end{align}
For more details see, \cite [section 4.2] {YY}.
\end{definition}
\begin{definition}  
 Let $X$ be a bicomplex module. Assume that $X_1,X_2$ are inner product spaces, with inner product $<.,.>_1, <.,.>_2$ respectively and corresponding norms $ \|.\|_1\;\text{and}\; \|.\|_2.$ Then the formula
 \begin{eqnarray*}
 <x,y>&=&<e_1x_1+e_2x_2,\;e_1y_1+e_2y_2>\\
 &=&e_1<x_1,y_1>_1+\;\;e_2<x_2,y_2>_2
\end{eqnarray*}
defines a bicomplex inner product on the bicomplex module\;$X=e_1X_1+e_2X_2.$
A bicomplex inner product in $\mathbb B \mathbb C$ can also be given by the formula  
$$<x,y>=x\;.\;y^{\dagger_3},$$
note that both these inner product coincides with each other. \\
Further, $<x,x>=e_1\|x_1\|^2_1+e_2\|x_2\|^2_2$ is a positive hyperbolic number, so it introduces a hyperbolic norm on $\mathbb{B}\mathbb{C}$-module $X$ consistent with $\mathbb{B}\mathbb{C}$-inner product,
i.e.,$$\|x\|_{\mathbb D}=\|e_1x_1+e_2x_2\|_{\mathbb D}=<x,x>^\frac{1}{2}.$$
Also bicomplex module $X$ with a real valued norm can be related to the bicomplex inner product in this way
\begin{eqnarray*}
\|x\|^2&=&\frac{1}{2}(<x_1,x_1>_1+<x_2,x_2>_2)\\ 
&=& \frac{1}{2}(\|x_1\|^2_1+\|x_2\|^2_2)
\end{eqnarray*}
For more details see, \cite  [section 4.3] {YY}.
\end{definition}
\begin{definition}
A bicomplex module $X$ with inner product $<.,.>$ is said to be bicomplex Hilbert space if it is complete with respect to the metric induced by its euclidean type norm generated by inner product. This is equivalent to say that $X$ is complete with respect to hyperbolic norm  generated by inner product. Further from the representation of $X=X_1e_1+X_2e_2$, it follows that $(X,<.,.>)$ is a bicomplex Hilbert space if and only if $(X_1,<,.,>_1)$ and $(X_2,<,.,>_2)$ are complex Hilbert spaces.  
 \end {definition}
 \begin{definition}
 Let $X$ and $Y$ be two $\mathbb{B}\mathbb{C}$-modules and let $T:X \rightarrow Y$ be a map such that 
 $$T(\lambda x+\mu y)=\lambda T(x)+\mu T(y),\;\forall\; x,y\in X,\; \forall\;\lambda,\mu \in \mathbb{B}\mathbb{C}.$$ 
 Then we say that $T$ is a $\mathbb{B}\mathbb{C}$-linear operator on $X$.\\
 Set $X_1=e_1X$ and $X_2=e_2X$, so that any $x\in X$ can be written as $$x=x_1e_1+x_2e_2.$$
Define the operators $T_l:X\rightarrow X_l$ as (see, \cite{CC})
$$T_l(x)=e_lT(x), l=1.2$$
Then we can write\begin{align} T=e_1T_1+e_2T_2 
\end{align}
so that the action of $T$ on $X$ can be decomposed as follows
$$T(x)=e_1T_1(x_1)+e_2T_2(x_2),\;\forall x\in X.$$			 
The decomposition (1.3) is called the $idempotent\;decomposition$ of the bicomplex linear operator $T$.
\end{definition}  

Consider the set $B(X,Y)$ of all bounded linear operators of $X$ into $Y$. For each $T\in B(X,Y)$, define a ${\mathbb D}$-valued norm on $T\in B(X,Y)$ (see, \cite{YY}, \cite{LL}) as follows:
$$\|T\|_{\mathbb D}=\text{sup}\left\{ \|T(x)\|_{\mathbb D}\;|\;x\in X,\; \|x\|_{\mathbb D}\leq 1\right\}.$$
Further, the operator $T$ is bounded if and only if the operators $T_1$ and $T_2$ are both bounded (see, \cite{CC}) and by using the idempotent decomposition (1.3) of $T$, the ${\mathbb D}$-valued norm on $T$ can be expressed as follows: 
 $$\|T\|_{\mathbb D}=e_1\|T_1\|_1+e_2\|T_2\|_2,$$ where $\|.\|_1$ and $\|.\|_2$ define the usual norms on $T_1$ and $T_2$ respectively. Moreover, norm on $T$ can also be given as    
  $$\|T\|=|\|T\|_{\mathbb D}|=\sqrt{\frac{\|T_1\|^2_1+\|T_2\|^2_2}{2}}$$ 
It is clear that the set $B(X,Y)$ is a $\mathbb{B}\mathbb{C}$ module. Further, it is easy to verify that if $Y$ is a Banach $\mathbb{B}\mathbb{C}$-module, so is $B(X,Y)$. In addition to this if $Y=\mathbb{B}\mathbb{C}$, then $B(X,\mathbb{B}\mathbb{C})$ is called the dual space of $X$ and is denoted by $X^{'}$.    
\begin{definition} Let $U\subseteq \mathbb{B}\mathbb{C}$ be an open set and let $U_i,i=1,2$ as in (1.1). Then $F:U\subseteq \mathbb{B}\mathbb{C}\rightarrow \mathbb{B}\mathbb{C}$ is a bicomplex holomorphic function if and only if there exists complex holmorphic functions $F_1$ and $F_2$ on $U_1$ and $U_2$, respectively, such that $$F(Z)=F(z+jw)=e_1F_1(z-iw)+e_2F_2(z+iw).$$ For further details on bicomplex holomorphic functions (see \cite{ZZ}, \cite{KK}, \cite{Z_1Z_1}).\\  
\end{definition}
G.B.Price book \cite{KK} contains an extensive survey of the various fundamental properties of bicomplex numbers and bicomplex function theory. Also \cite{Z_1Z_1} contains detail information on bicomplex holomorphic functions. In the last few years, Lavoie, Marchildon and Rochon (see, \cite{G_1G_1}, \cite{GG}) have introduced finite and infinite dimensional bicomplex Hilbert spaces and studied some of their basic properties. In \cite{LL}, the concept of bicomplex topological modules and the fundamental theorems of functional analysis to the framework of bicomplex topological modules are introduced. Recently, D. Alpay et.al. \cite{YY} have given a nice and clear survey of bicomplex functional analysis and also discussed many new ideas as well as results. In \cite{HH}, it is shown that the spectrum of bicomplex bounded linear operator is unbounded. Also in extended version of the paper \cite{UU}, it is shown that the point spectrum of bicomplex bounded linear operator on $l^2(\mathbb B \mathbb C)$ is equal to its null cone. In \cite{ff}, the concept of the Cauchy-Kowalewski product for bicomplex holomorphic functions is discussed.  
\end{section}
\begin{section} 
  {Operators on Bicomplex-Hilbert Spaces}
In this section we characterize normal and unitary operators on bicomplex Hilbert spaces. We also discuss parallelogram law for bicomplex Hilbert spaces.
\begin{definition} Let $A=[a_{ij}]$ be a ${n\times n}$ matrix over $\mathbb {B}\mathbb {C}$. Then $A$ is said to be normal matrix if
$$ AA^{t\dagger_3}=A^{t\dagger_3}A,$$ where $A^{t\dagger_3}=(A^t)^{\dagger_3}=(A^{\dagger_3})^t=[a^{\dagger_3}_{ji}]$ is the transpose $\dagger_3$- conjugate of A.  
\end{definition}                       
\begin{proposition}   Let $A$ be a ${n\times n}$ matrix over $\mathbb {B}\mathbb {C}$. Let $A=A_1e_1+A_2 e_2$ be its idempotent decomposition.Then $A$ is normal if and only if its idempotent components  $A_1,A_2$ are complex ${n\times n}$ normal matrices.
\end{proposition} 
\begin{proof}  Let $A$ be a ${n\times n}$ matrix over $\mathbb{B}\mathbb{C}$.
Then
\begin{eqnarray*} 
\text {A is normal} & \Leftrightarrow & AA^{\dagger_3t}=A^{\dagger_3t}A\\
& \Leftrightarrow &(A_1e_1+A_2e_2){(A_1e_1+A_2e_2)}^{\dagger_3t}={(A_1e_1+A_2e_2)}^{\dagger_3t}(A_1e_1+A_2e_2)\\                         
& \Leftrightarrow & A_1\overline A_1^te_1+A_2\overline A_2^te_2=\overline A_1^tA_1e_1+\overline A_2^tA_2e_2\\         
& \Leftrightarrow & A_1\overline A_1^t=\overline A_1^tA_1\; \text{and}\; A_2\overline A_2^t=\overline A_2^tA_2. 
 \end{eqnarray*}
                           
  Hence $A$ is normal matrix if and only if $A_1,A_2$ are normal matrices.
 \end{proof}
 \begin{proposition}  Let  $A=B+jC$ be a ${n\times n}$ matrix over $\mathbb{B}\mathbb{C}$.Then $A$ is normal matrix if and only if its cartesian components $B,C$ are commutative hermitian complex matrices.
 \end{proposition}
\begin{proof} Let $A$ be a  ${n\times n}$ matrix over $\mathbb{B}\mathbb{C}$. Then
\begin{eqnarray*}
\text {A is normal} & \Leftrightarrow & A^{\dagger_3t} A=A A^{\dagger_3t}\\
& \Leftrightarrow & (\overline B^t-j\overline C^t)(B+jC)=(B+jC)(\overline B^t-j\overline C^t)\\
& \Leftrightarrow & \overline B^t B+\overline C^t C+j\overline B^t C-j\overline C^t B=B\overline B^t+C\overline C^t-jB\overline C^t+jC\overline B^t\\
& \Leftrightarrow & \overline B^t=B,\;\overline C^t=C\; \text{and}\; BC=CB.
 \end{eqnarray*}
Hence $A$ is normal matrix if and only if $B,C$ are commutative hermitian matrices.
\end{proof}
\begin{definition} Let $X$ and $Y$ be two bicomplex Hilbert spaces. Then the bicomplex adjoint operator $T^*:Y \rightarrow X$ for a bounded operator $T:X \longrightarrow Y$ is defined by the equality 
 $$<T(x),y>=<x,T^*(y)>.$$
 \end{definition}
\begin{remark} Bicomplex adjoint $T^*$ can also be represented as (see, \cite{YY})                 
 $$T^{*}=e_1T^{*}_1+e_2T^{*}_2,$$
  where $T^{*}_1 \; \text{and}\; T^{*}_2$ are the complex adjoints of the operators $T_1 \; \text{and}\; T_2$ respectively.
\end{remark}
\begin{definition} Let $X$ be a bicomplex Hilbert space. An operator $T\in B(X)$ is said to be a bicomplex self-adjoint operator if $T=T^{*}$ (see, \cite {XX}).
\end{definition}

 \begin{definition} Let $X$ be a bicomplex Hilbert space. An operator $T\in B(X)$ is said to be a bicomplex normal operator if $TT^{*}=T^{*}T$, where $T^{*}$denotes the adjoint of $T$.
  \end{definition}
\begin{definition} Let $X$ be a bicomplex Hilbert space. An operator $T \in B(X)$ is said to be a bicomplex unitary if $TT^{*}=T^{*}T=I$, where $T^{*}$denotes the adjoint of $T$ and $I$ is the identity operator on $X$.
\end{definition} 
\begin{definition} An operator $T$ on a bicomplex Hilbert space $H$  is said to be a zero operator if $T(x)=0$, $\forall\; x\in H.$ 
\end{definition}
\begin{remark}
$(i)$ $T$ is a zero operator on $H$ if and only if $T_1$\;and\;$T_2$ are zero operators on $H_1$\;and\;$H_2$ respectively.\\
$(ii)$ If $T\in B(H)$ and if  $(Tx,x)=0$, for every $x\in H,$ then $T=0.$
\end{remark}
\begin{definition}
A bicomplex self-adjoint operator $T$ on a bicomplex Hilbert space $X$ is said to be positive operator if $<Tx,x>$ is a positive hyperbolic number, for every $x\in X.$
\end{definition}
\begin{example}  For any operator $T$, the products $TT^{*}\; \text{and}\;\; T^{*}T$ are positive operators.
\end{example}
The following proposition is easy to prove:
\begin{proposition}
 Let $T \in B(X)$ such that $T=e_1T_1+e_2T_2$ be its idempotent decomposition. Then the following holds:\\ 
$(i)$ $T$ is a bicomplex self adjoint operator on $X$ if and only $T_1, T_2$ are complex self adjoint operators on $X_1, X_2$ respectively.\\   
$(ii)$ $T$ is a bicomplex normal operator on $X$ if and only if $T_1,T_2$ are complex normal operators on $X_1,X_2$ respectively.\\  
$(iii)$ $T$ is a bicomplex unitary operator on $X$ if and only if $T_1,T_2$ are complex unitary operators on $X_1,X_2$ respectively.\\
$(iv)$ $T$ is a bicomplex self adjoint operator on $X$ if and only if $<Tx,x>$ is a hyperbolic number, $\forall x\in X.$
\end{proposition}

\begin{theorem}\textbf{(Parallelogram Law)}: Let $X$ be a bicomplex Hilbert space and $x\;\text{and}\;y \in X.$ Then\\ 
  $\| x+y \|^2_{\mathbb D} + \|x-y \|^2_{\mathbb D}= 2(\|x\|^2_{\mathbb D} + \|y\|^2_{\mathbb D}).$
 \end{theorem}
 \begin{proof} For any $x,y \in X,$ 
 \begin{eqnarray*}
 \|x+y\|^2_{\mathbb D} &= &<x+y,x+y>\\
 &= &<x_1e_1+x_2e_2+y_1e_1+y_2e_2,x_1e_1+x_2e_2+y_1e_1+y_2e_2>\\
 & = &<(x_1+y_1)e_1 + (x_2+y_2)e_2,(x_1+y_1)e_1 + (x_2+y_2)e_2>\\
 & = & e_1<(x_1+y_1),(x_1+y_1)>_1 +e_2 <(x_2+y_2),(x_2+y_2)>_2\\
 & = & e_1\|x_1\|^2_1 + e_1\|y_1\|^2_1 + e_1<x_1,y_1>_1 + e_1<y_1,x_1>_1\\
 &  & +e_2\|x_2\|^2_2 + e_2\|y_2\|^2_2 + e_2<x_2,y_2>_2 + e_2<y_2,x_2>_2
 \end{eqnarray*} 
\begin{eqnarray*}\;\text{Also},\;
\|x-y\|^2_{\mathbb D}&=&<x-y,x-y>\\
&=&<x_1e_1+x_2e_2-(y_1e_1+y_2e_2),x_1e_1+x_2e_2-(y_1e_1+y_2e_2)>\\
&=&<(x_1-y_1)e_1 + (x_2-y_2)e_2,(x_1-y_1)e_1 + (x_2-y_2)e_2>\\
& = & e_1<(x_1-y_1),(x_1-y_1)>_1 +e_2 <(x_2-y_2),(x_2-y_2)>_2\\
&=&e_1\|x_1\|^2_1-e_1<x_1,y_1>_1-e_1<y_1,x_1>_1+e_1\|y_1\|^2_1+e_2\|x_2\|^2_2\\
& &-e_2<x_2,y_2>_2-e_2<y_2,x_2>_2+e_2\|y_2\|^2_2
\end{eqnarray*}
On adding we get, \begin{eqnarray*} \|x+y\|^2_{\mathbb D}+\|x-y\|^2_{\mathbb D}&=&2e_1\|x_1\|^2_1+2e_1\|y_1\|^2_1+2e_2\|x_2\|^2_2+2e_2\|y_2\|^2_2\\
&=&2(e_1\|x_1\|^2_1+e_2\|x_2\|^2_2)+2(e_1\|y_1\|^2_1+e_2\|y_2\|^2_2)\\
&=&2(\|x\|^2_{\mathbb D} + \|y\|^2_{\mathbb D}).
 \end{eqnarray*}
Now consider the euclidean(real valued) norm, in this case,
 \begin{eqnarray*}
 \|x+y\|^2&=& \|(x_1+y_1)e_1+(x_2+y_2)e_2\|^2\\
&=& \frac{1}{2}\left(<x_1+y_1,x_1+y_1>_1+<x_2+y_2,x_2+y_2>_2\right)\\
&=& \frac{1}{2}\left(\|x_1+y_1\|^2_1+\|x_2+y_2\|^2_2 \right).
\end{eqnarray*}
Similarly, \;$\|x-y\|^2 = \frac{1}{2}\left( \|x_1-y_1\|^2_1+\|x_2-y_2\|^2_2 \right).$
On adding we get, 
\begin{eqnarray*}
\|x+y\|^2 + \|x-y\|^2& =& \frac{1}{2} \left(\|x_1+y_1\|^2_1+\|x_2+y_2\|^2_2 + \|x_1-y_1\|^2_1+\|x_2-y_2\|^2_2 \right).\\
&=&  \frac{1}{2} \left(\|x_1+y_1\|^2_1+ \|x_1-y_1\|^2_1+\|x_2+y_2\|^2_2+\|x_2-y_2\|^2_2 \right).\\
& &\text{using the Parallelogram law for the complex Hilbert spaces}\\
&=&  \frac{1}{2} \left(2(\|x_1\|^2_1+\|y_1\|^2_1) + 2(\|x_2\|^2_2+\|y_2\|^2_2)\right).\\
&=& 2 \left(\frac{1}{2} (\|x_1\|^2_1+\|x_2\|^2_2) + \frac{1}{2} (\|y_1\|^2_1+\|y_2\|^2_2)\right).\\
&=& 2 \left(\|x\|^2 + \|y\|^2 \right).
\end{eqnarray*} 
\end{proof}
\begin{theorem} Let $X$ be a bicomplex Hilbert space and $T\in B(X)$. Then $T$ is a normal operator on $X$ if and only if $\|Tx\|_\mathbb D=\|T^{*}x\|_\mathbb D,\forall\; x\in X.$
\end{theorem}
\begin{proof} Since $T$ is a bicomplex bounded linear operator on $X$, then    
\begin{eqnarray*}
\|T^{*}x\|_\mathbb D =\|Tx\|_\mathbb D 
&\Leftrightarrow& \|T^{*}x\|^2_\mathbb D=\|Tx\|^2_\mathbb D\\    
&\Leftrightarrow& <T^{*}x,T^{*}x> = <Tx,Tx>\\
&\Leftrightarrow& e_1<T^{*}_1x_1,T^{*}_1x_1>_1 + e_2<T^{*}_2x_2,T^{*}_2x_2>_2\\
& = &e_1<T_1x_1,T_1x_1>_1 + e_2<T_2x_2,T_2x_2>_2\\
&\Leftrightarrow& e_1<T_1T_1^{*}x_1,x_1>_1 + e_2<T_2T^{*}_2x_2,x_2>_2\\
& = &e_1<T^{*}_1T_1x_1,x_1>_1 + e_2<T^{*}_2T_2x_2,x_2>_2\\
&\Leftrightarrow& T_1T^{*}_1=T^{*}_1T_1\; \;\text{and}\;\; T_2T^{*}_2=T^{*}_2T_2\\
&\Leftrightarrow& TT^{*}=T^{*}T.
\end{eqnarray*}
This implies that $T$ is normal.\\
\end{proof}
\begin{remark} \; Using $(1.2)$,  we get  $|\|Tx\|_\mathbb D|=|\|T^{*}\|_\mathbb D|$,\\\\
i.e.,~~~~~~~~~~~~~~~~~~~~~~~~~~~~~~~~~~~~~~~~  $\|Tx\|=\|T^{*}x\|.$ 	
\end{remark}

\begin{theorem} Let $T\in B(X)$. Then the following properties are true for the adjoint operators:\\
$(i)$ $\|T^{*}\|_\mathbb D=\|T\|_\mathbb D$,\\
$(ii)$ $\|T^{*}T\|_\mathbb D=\|T\|^2_\mathbb D$.
\end{theorem}
\begin{proof}
 (i) trivially holds.\\
  (ii) By using the idempotent decomposition of $T$ and $T^{*}$, we can write 
  \begin{align} \|T^{*}T\|_\mathbb D& = \|(e_1T^{*}_1+e_2T^{*}_2)(e_1T_1+e_2T_2)\|_\mathbb D\notag\\
& =  \|e_1T^{*}_1T_1+e_2T^{*}_2T_2\|_\mathbb D\notag\\
& =  e_1\|T^{*}_1T_1\|_1+e_2\|T^{*}_2T_2\|_2\notag\\
& =  e_1\|T_1\|^2_1+e_2\|T_2\|^2_2\notag\\
& =  \|T\|^2_\mathbb D.\notag
\end{align}
\end{proof}
\begin{corollary} Let $T\in B(X)$. Then $\|T\|^2\leq\|T^{*}T\|\leq\sqrt{2}\; \|T\|^2.$
\end{corollary}
\begin{proof} Since   $\|T^{*}T\|  \leq  \sqrt{2}\|T^{*}\|\|T\| = \sqrt{2}\|T\|^2$.\\ 
Also using the idempotent decomposition of $T$, we have
\begin{eqnarray*} \|T\|^2&=&\|T_1e_1+T_2e_2\|^2\\
&=&\frac{1}{2}\left( \|T_1\|^2_1+\|T_2\|^2_2\right)\\
&=&\frac{1}{2}\left( \|T^{*}_1T_1\|_1+\|T^{*}_2T_2\|_2\right)\\
&\leq&\frac{1}{\sqrt{2}}\left( \sqrt{ \|T^{*}_1T_1\|^2_1+\|T^{*}_2T_2\|^2_2}\right)\\
&=&\|T^*T\|.
\end{eqnarray*}
Hence,\;\; $\|T\|^2\leq\|T^{*}T\|\leq\sqrt{2}\;\|T\|^2.$
\end{proof}
\begin{proposition} If {T} is a normal linear operator on a bicomplex Hilbert space $X.$ Then\  $$\|T^2\|_\mathbb D=\|T\|^2_\mathbb D.$$
\end{proposition}
\begin{proof} We know that $T=T_1e_1+T_2e_2$ is a normal operator if and only if $T_1\;\text{and}\;T_2$ are complex normal operators.\\
Also for complex normal operators $T_1$\; and\; $T_2$, we have $$\|T^2_1\|_1=\|T_1\|^2_1\; \;\text{and}\;\; \|T^2_2\|_2=\|T_2\|^2_2.$$
 \begin{eqnarray*} \text{Thus,}\;\;\;
\|T^2\|_\mathbb D&=&\|(T_1e_1+T_2e_2)^2\|_\mathbb D\\
&=& \|T^2_1e_1+T^2_2e_2\|_\mathbb D\\
&=&\|T^2_1\|_1e_1+\|T^2_2\|_2e_2\\
&=& \|T_1\|^2_1e_1+\|T_2\|^2_2e_2\\
&=&\|T\|^2_\mathbb D.
\end{eqnarray*}
\end{proof}
\begin{remark} For a normal linear operator $T$ on a bicomplex Hilbert space $X$, we have $$\|T\|^2\leq\|T^2\|\leq\sqrt{2}\;\|T\|^2.$$
\end{remark}
\end{section}
\begin{section} {Applications of Hahn-Banach Theorem and Bicomplex $C^\textbf{*}$-Algbras}
In this section we introduce quotient modules, bicomplex $C^\textbf{*}$-algebra, annihilators and discuss a couple of applications of Hahn-Banach theorem. We also describe the duals of a submodule $M$ and of $X/M$ with the aid of the annihilator $M^\bot$ of $M.$\\
We can restate bicomplex Hahn Banch theorem (see \cite{LL}) as:
\begin{theorem} Let $Y$ be a submodule of a $\mathbb{B}\mathbb{C}$-normed module $X$ and let $y^{'}\in Y^{'}.$ Then $x^{'}$ is the bicomplex extension of $y^{'}$ if and only if $x^{'}_i$ is the complex extension of $y^{'}_i$, for $i=1,2$.
\end{theorem}
\begin{definition} Let $M$ be a submodule of a $\mathbb{B}\mathbb{C}$-module $X$.\\
Then we can write $X=X_1e_1+X_2e_2$, where $X_1,X_2$ are complex linear spaces and  $M=M_1e_1+M_2e_2$, where $M_1\;\text{and}\;M_2$ are complex linear subspaces of $X_1\;\text{and}\;X_2$ respectively, so that $X_i/M_i$ for i=1,2 are quotient spaces over the complex field.\\ 
Consider the set $X/M=\left\{M+x:x\in X\right\}$ , where $M+x$ is a coset of $M$ that contains $x$.\\
Let $x,y \in X$ and $\alpha \in \mathbb{B}\mathbb{C}$. Then define the following operations on $X/M$ as
\begin{align}\text{(i)}\;(M+x)&+(M+y)\notag\\
&=\left((M_1e_1+M_2e_2)+(x_1e_1+x_2e_2)\right)+\left((M_1e_1+M_2e_2)+(y_1e_1+y_2e_2)\right)\notag\\
&=(M_1+x_1)e_1+(M_2+x_2)e_2+(M_1+y_1)e_1+(M_2+y_2)e_2\notag\\
&=(M_1+x_1+y_1)e_1+(M_2+x_2+y_2)e_2\notag\\
&= (M+x+y)\notag
\end{align}
\begin{align}
\text{(ii)}\; \alpha(M+x)&=\alpha_1e_1+\alpha_2e_2\left((M_1e_1+M_2e_2)+(x_1e_1+x_2e_2)\right)~~~~~~~~~~~~~~~~~~~~~~~~~~~~\notag\\
&=\alpha_1e_1+\alpha_2e_2\left((M_1+x_1)e_1+(M_2+x_2)e_2\right)\notag\\
&=\alpha_1(M_1+x_1)e_1+\alpha_2(M_2+x_2)e_2\notag\\
&=(M_1+\alpha_1x_1)e_1+(M_2+\alpha_2x_2)e_2\notag\\
&=M_1e_1+M_2e_2+(\alpha_1e_1+\alpha_2e_2)(x_1e_1+x_2e_2)\notag\\
&=M+\alpha x.\notag
\end{align}

With the operations defined above $X/M$ form a module over $\mathbb{B}\mathbb{C}$ and is called bicomplex quotient module.
\end{definition}
\begin{remark} For any $x\in X$, $M+x=(M_1+x_1)e_1+(M_2+x_2)e_2$, so one can conclude  
$$ X/M = e_1X_1/M_1+e_2X_2/M_2$$
\end{remark}
\begin{definition}Let $M$ be a submodule of a $\mathbb{B}\mathbb{C}$-module $X$ such that $X/M$ form a bicomplex quotient module. Define a mapping $f:X\longrightarrow X/M$ as $f(x)=M+x$.\\
Clearly, $f$ is a bicomplex linear mapping and is called a bicomplex quotient map of $X$ onto $X/M$.\\ 
Let $X$ be a normed module and $f$ be a quotient map of $X$ onto $X/M$. Then the quotient norm on $X/M$ is defined as
$$\|f(x)\|_\mathbb D=inf\left\{\|x-y\|_\mathbb D\;:\;y\in M\right\}.$$
\end{definition} 
\begin{remark}\;\begin{eqnarray*}\text{Since}\;\;\;\; inf\left\{\|x-y\|_\mathbb D\right\}&=&inf\left\{\|x_1-y_1\|_1e_1+\|x_2-y_2\|_2e_2\right\}\\
&=&e_1\;inf\left\{\|x_1-y_1\|_1\right\}+e_2\;inf\left\{\|x_2-y_2\|_2\right\}.
\end{eqnarray*}
Thus, the quotient norm on $X/M$ can be given as $$\|f(x)\|_\mathbb D=e_1\|f_1(x_1)\|_1+e_2\|f_2(x_2)\|_2.$$
\end{remark}
\begin{definition}Let $X$ be a $\mathbb{B}\mathbb{C}$-Banach module and $M$ be a submodule of $X$. Then the annihilator of $M$ is defined as
 $$M^\bot=\left\{x^{'}\in X^{'}:\;<x,x^{'}>=0,\; \forall\; x\in M\right\}. $$
\end{definition}
\begin{align} \text{Further,}\;x^{'}\in M^\bot&\Leftrightarrow\;<x,x^{'}>=0 , \forall\; x\in M\notag\\
&\Leftrightarrow\; <x_1,x^{'}_1>_1e_1+<x_2,x^{'}_2>_2e_2=0, \forall\; x_1e_1+x_2e_2\in M\notag\\
&\Leftrightarrow\; <x_1,x^{'}_1>_1=0\;\;\text{and}\; <x_2,x^{'}_2>_2=0, \forall x_1\in M_1,\forall x_2\in M_2\notag\\
&\Leftrightarrow\; x^{'}_1\in M^\bot_1\;\;\text{and}\; \;\; x^{'}_2\in M^\bot_2.\notag
\end{align}
Thus, one can concludes that	annihilator of bicomplex submodule $M$ is equal to the annihilator of its idempotent components $M_1\;\text{and}\; M_2,$  i.e.,
$$M^\bot=e_1M^\bot_1+e_2M^\bot_2.$$
\begin{theorem} Let $M$ be a closed submodule of a $\mathbb{B}\mathbb{C}$-Banach module $X$.
Then\\
$(a)$ $M^{'}$ is an  isometrically isomorphic to $X^{'}/M^\bot$.\\ 
$(b)$ $(X/M)^{'}$ is isometrically isomorphic to $M^\bot$.
\end{theorem}
\begin{proof}(a)  Since every $\mathbb{B}\mathbb{C}$-linear functional $f$ on $M$ can be written as\\
$f=f_1e_1+f_2e_2$, where $f_i$ are complex-linear functional on normed linear spaces $M_i$, for $i=1,2$.\\
Thus, one can write $M^{'}=M^{'}_1e_1+M^{'}_2e_2$.  
 \begin{eqnarray*} \text{Further},\;\; X^{'}/M^\bot&=&\left\{x^{'}+M^\bot:x^{'}\in X^{'}\right\}\\ 
 &=&\left\{x^{'}_1e_1+x^{'}_2e_2+M^\bot_1e_1+M^\bot_2e_2:x^{'}\in X^{'}\right\}\\
 &=&\left\{(x^{'}_1+M^\bot_1)e_1+(x^{'}_2+M^\bot_2)e_2:x^{'}\in X^{'}\right\}\\
 &=& e_1\left\{(x^{'}_1+M^\bot_1):x^{'}_1\in X^{'}_1\right\}+e_2\left\{(x^{'}_2+M^\bot_2):x^{'}_2\in X^{'}_2\right\}\\
 &=&e_1\; X^{'}_1/M^\bot_1+\;e_2\; X^{'}_2/M^\bot_2.
\end{eqnarray*}  
Also $M_i$ is a closed linear subspace of a complex Banach space $X_i$ for $i=1,2$ and thus by using \cite [Theorem 4.9]{WW}, there exist a linear mapping $$f_i:M^{'}_i\longrightarrow X^{'}_i/M^\bot_i$$ defined as $f_i(m^{'}_i)=(x^{'}_i+M^\bot_i)$ such that $f_i$ is an isometric isomorphism, where $x^{'}_i$ is an extension of $m^{'}_i$ for $i=1,2$\\  
Define $f=f_1e_1+f_2e_2$ as $f(m^{'})=x^{'}+M^\bot$, where $x^{'}$ is a bicomplex Hahn Banach extension of $m^{'}$.\\
Clearly, $f:M^{'}\longrightarrow X^{'}/M^\bot$ is a $\mathbb{B}\mathbb{C}$-linear as well as bijective map. Moreover, isometry follows from the isometries of $f_1$ and $f_2$,\;\;i.e.,
\begin{eqnarray*} \|f(m^{'})\|_\mathbb D&=&\|f_1(m^{'}_1)e_1+f_2(m^{'}_2)e_2\|_\mathbb D\\
&=&e_1\|f_1(m^{'})\|_1+e_2\|f_2(m^{'}_2)\|_2\\
&=&e_1\|m^{'}_1\|_1+e_2\|m^{'}_2\|_2\\
&=& \|m^{'}\|_\mathbb D.
\end{eqnarray*}
 (b) By using remark (3.3), we can write, 
$$X/M=e_1X_1/M_1+e_2X_2/M_2.$$
Also, by the idempotent decomposition of bicomplex linear functional $f$ over $X/M$, one can write 
$$(X/M)^{'}=e_1(X_1/M_1)^{'}+e_2(X_2/M_2)^{'}$$
Now $M_i$ is a closed subspace of a complex Banach space $X_i$ and by using \cite [Theorem 4.9]{WW}, there exist a linear mapping $$f_i:(X_i/M_i)^{'}\longrightarrow M^\bot_i$$ defined as   
$f_i(y^{'}_i)=y^{'}_i(g_i),$ where $g_i:X_i\longrightarrow X_i/M_i$ be the quotient map for $i=1,2$.\\
Define $f=f_1e_1+f_2e_2$ as $f(y^{'})=y^{'}(g)$, where $g=g_1e_1+g_2e_2:X\longrightarrow X/M$ be the quotient map.\\ 
Clearly $f:(X/M)^{'}\longrightarrow M^\bot$ is a $\mathbb{B}\mathbb{C}$-linear map which is also bijective.\\
Also isometry of $f$ follows from the isometries of $f_1$ and $f_2$,\;\; i.e.,
 \begin{eqnarray*} \|f(y^{'})\|_\mathbb D&=&\|f_1(y^{'}_1)e_1+f_2(y^{'}_2)e_2\|_\mathbb D\\
&=& e_1\|f_1(y^{'}_1)\|_1+ e_2\|f_2(y^{'}_2)\|_2\\
&=& e_1\|y^{'}_1\|_1+e_2\|y^{'}_2\|_2\\
&=&\|y^{'}\|_\mathbb D.
\end{eqnarray*} 
\end{proof}
 \begin{remark}
   We can rewrite the above theorem as:\\
(a) $M^{'}$ is isometrically isomorphic to $X^{'}/M^\bot$ if and only if $M^{'}_i$ is isometrically isomorphic to $X^{'}_i/M^\bot_i$\; for $i=1,2$.\\
(b) $(X/M)^{'}$ is isometrically isomorphic to $M^\bot$ if and only if $(X_i/M_i)^{'}$ is isometrically isomorphic to $M^\bot_i$ \;for $i=1,2$. 
 \end{remark}
 \begin{theorem}Let $X$ be a bicomplex normed module. Then $X=e_1X_1+e_2X_2$ is a bicomplex Banach module if and only if $X_1$ and $X_2$ are complex Banach spaces.
 \end{theorem}
\begin{proof} Firstly suppose that $X=e_1X_1+e_2X_2$ is a bicomplex Banach module. To show that $X_i$ is a complex Banach space, let $\left\{a_{ni}\right\}_{n=0}^\infty$ be a Cauchy sequence in $X_i$ for $i=1,2$, where $\forall\; n \in \mathbb N,\; a_{ni}=e_ia_n$. Thus, $\left\{a_{n}\right\}_{n=0}^\infty$ is a Cauchy sequence in $X$. But $X$ is a Banach module, so for given $\epsilon \geq 0$ and $a\in X$, there exist $r\in \mathbb N$ such that $\|a_n-a\|\leq\epsilon \;;\;\forall\;n\geq r.$\\     
Now, $\|e_ia_n-e_ia\|=\|e_i(a_n-a)\|\leq\; \sqrt{2}\; \|e_i\| \|a_n-a\|=\sqrt{2}\frac{1}{\sqrt{2}}\|a_n-a\|\leq \epsilon,\;\forall\;n\geq r.$\\
i.e., $e_ia_n\rightarrow e_ia$, for $i=1,2$.\\
Hence $X_i$ for $i=1,2$ is a complex Banach space.\\
To prove the converse part, let $\left\{a_{n}=e_1a_n+e_2a_n\right\}_{n=0}^\infty$ be  Cauchy sequence in $X$, then $\left\{e_ia_{n}\right\}_{n=0}^\infty$  is a Cauchy sequence in $X_i$, for $i=1,2$. By using the completeness of $X_i$, it is easy to show that $X$ is complete. Hence $X$ is a bicomplex Banach module.
\end{proof}

\begin{definition} An algebra $A$ over $\mathbb{B}\mathbb{C}$ that has a norm $\|.\|$ relative to which $A$ is a Banach space and such that for every  $x,y$ in $A$,   
$$ \|xy\|\leq \sqrt{2}\;\|x\|\|y\|$$ 
is called a bicomplex Banach algebra which is clearly a generalisation of classical Banach algebra.
\end{definition} 
\begin{definition} A mapping $x\longrightarrow x^\textbf{*}$ of a bicomplex Banach algebra $A$ into $A$ is called an involution on $A$ if the following properties hold for $x,y\in A$ and $\alpha \in \mathbb B \mathbb{C}$:
\begin{enumerate}

\item[(i)] $(x^\textbf{*})^\textbf{*}=(x)$
\item[(ii)] $(xy)^\textbf{*}=y^\textbf{*} x^\textbf{*}$
\item[(iii)] $(\alpha x+y)^\textbf{*}=\alpha^{\dagger_3} x^\textbf{*}+y^\textbf{*}$.
\end{enumerate}
\end{definition}
 
\begin{definition} A bicomplex Banach algebra with an involution on it is called a bicomplex $B^\textbf{*}$- algebra.
\end{definition}
\begin{definition} 
A bicomplex $B^\textbf{*}$- algebra such that for every $x$ in $A$,
$$\|x\|^2\leq\|x^\textbf{*}x\|\leq \sqrt{2}\|x\|^2$$ is called a bicomplex $C^\textbf{*}$- algebra.
\end{definition}
\textbf{Examples}\\
(1) If $X$ is a bicomplex Hilbert space, then $B(X)$ is a bicomplex $C^\textbf{*}$- algebra, where for each $T$ in $B(X)$, $T^\textbf{*}$ is the adjoint of $T$.\\
(2)  Let $X$ be a bicomplex compact space and $C(X)$ denotes the space of all bicomplex-valued continuous functions on $X$. Then $C(X)$ is a bicomplex $C^\textbf{*}$-algebra,where $f^\textbf{*}(x)=f^{\dagger_3}(x)$, $\forall\; f \in C(X)$,\;$\forall \;x \in X$. 
\end{section}
\begin{section} { Examples of Bicomplex Function Spaces }
Let $\mathbb{D}_{\mathbb{B}\mathbb{C}}=\left\{Z=z+jw=e_1z_1+e_2z_2\;|\;(z_1,z_2)\in \mathbb{D}^2\right\}$
 be the unit discus in $\mathbb{B}\mathbb{C}$,
where $\mathbb{D}$ is the unit disk in the complex plane and $\mathbb{D}^2=\mathbb{D}\times\mathbb{D}$. Then the bicomplex Hardy space $H^2(\mathbb{D}_{\mathbb{B}\mathbb{C}})$ (see $\cite{YY}$) is defined to be the set of all holomorphic functions $f:\mathbb{D}_{\mathbb{B}\mathbb{C}}\rightarrow \mathbb{B}\mathbb{C}$ such that its sequence of power series coefficients is square-summable, i.e.,
$$H^2(\mathbb{D}_{\mathbb{B}\mathbb{C}})=\left\{f(Z)=\sum_{n=0}^\infty{a_nZ^n} \;\text{holomorphic in}\; \mathbb{D}_{\mathbb{B}\mathbb{C}}:\sum_{n=0}^\infty {|a_n|^2_k}\; \text{is convergent}\right\},$$ where $\forall n\in \mathbb N, a_n \in \mathbb B \mathbb C$.\\
By setting $a_n=e_1a_{n1}+e_2a_{n2}$, we find that both the complex series $$\sum_{n=0}^\infty |a_{n1}|^2,\;\;\;\;		  \sum_{n=0}^\infty |a_{n2}|^2$$ are convergent, so that one can write the bicomplex Hardy space as $$ H^2(\mathbb{D}_{\mathbb{B}\mathbb{C}})=e_1H^2(\mathbb{D})+e_2H^2(\mathbb{D})$$ where $H^2(\mathbb{D})$ denotes the Hardy space of the unit disk $\mathbb{D}$ in the complex plane.\\
The $\mathbb{B}\mathbb{C}$-valued inner product on $H^2(\mathbb{D}_{\mathbb{B}\mathbb{C}})$  is given by $$<f,g>_{H^2(\mathbb{D}_{\mathbb{B}\mathbb{C}})}=\sum_{n=0}^\infty{a_n\;.\; b_n^{\dagger_3}}$$
and it generates the $\mathbb{D}$-valued norm:
\begin{eqnarray*}
\|f\|^2_{\mathbb{D},H^2(\mathbb{D}_{\mathbb{B}\mathbb{C}})}&=&<f,f>_{H^2(\mathbb{D}_{\mathbb{B}\mathbb{C}})}\;\;=\;\;\sum_{n=0}^\infty|a_n|^2_k\\
&=&e_1\sum_{n=0}^\infty|a_{n1}|^2+e_2\sum_{n=0}^\infty|a_{n2}|^2\\
&=&e_1\|f_1\|^2_{H^2(\mathbb{D})}+e_2\|f_2\|^2_{H^2(\mathbb{D})}.
\end{eqnarray*}
Thus we can say a bicomplex holomorphic function $f\in H^2(\mathbb{D}_{\mathbb{B}\mathbb{C}})$ if and only if $f_1,f_2 \in H^2(\mathbb{D})$ such that 
$$f(z_1+jz_2)=e_1f_1(z_1-iz_2)+e_2f_2(z_1+iz_2),$$
 for all $z_1+jz_2\in \mathbb{D}_{\mathbb{B}\mathbb{C}}$.   
\begin{definition}Consider a sequence $\left\{\beta(n)\right\}_{n=0}^\infty$ of positive hyperbolic numbers with
 $$\beta(0)=e_1+e_2=1 \;\text{and}\; \text{lim}_{{n} \rightarrow \infty}\beta(n)^\frac{1}{n}\geq e_1+e_2=1.$$
We now define the {\it weighted sequence space} $l^2_\beta({\mathbb{B}\mathbb{C}})$ of bicomplex numbers as 
$$ l^2_\beta({\mathbb{B}\mathbb{C}})=\left\{\left\{a_n\right\}_{n=0}^\infty: a_n\in \mathbb{B}\mathbb{C}\;\text{and}\; \sum_{n=0}^\infty |a_n\beta(n)|^2_k\; \text{is convergent}\right\},$$  
where the $\mathbb{D}$-valued norm on $\left\{a_n\right\}_{n=0}^\infty$ is defined to be
$$\|\left\{a_n\right\}\|_\mathbb{D}=\left(\sum_{n=0}^\infty|a_n\beta(n)|^2_k\right)^\frac{1}{2}.$$ 
Setting $a_n=a_{n1}e_1+a_{n2}e_2$,\;$\beta(n)=\beta_1(n)e_1+\beta_2(n)e_2$, one gets:
$$|a_n\beta(n)|^2_k=|a_{n1}\beta_1(n)|^2 e_1+|a_{n2}\beta_2(n)|^2e_2,$$
which means that both the complex series $$\sum_{n=0}^\infty{|a_{n1}\beta_1(n)|^2},\;\;\;\; \sum_{n=0}^\infty{|a_{n2}\beta_2(n)|^2}$$ are convergent and thus both the sequences $\left\{a_{n1}\right\}_{n=0}^\infty$ and $\left\{a_{n2}\right\}_{n=0}^\infty$ belong to the weighted Hardy space $l^2_{\beta1}$ and $l^2_{\beta2}$ respectively.
Thus, bicomplex weighted sequence space can be written as 
 $$ l^2_\beta({\mathbb{B}\mathbb{C}})=l^2_{\beta1}e_1+ l^2_{\beta2}e_2$$
\end{definition}
\begin{definition}
 The {\it bicomplex weighted Hardy space} is defined as
 $$H^2_\beta(\mathbb{D}_{\mathbb{B}\mathbb{C}})=\left\{f(Z)=\sum_{n=0}^\infty{a_nZ^n} \;\text{holomorphic in}\; \mathbb{D}_{\mathbb{B}\mathbb{C}}:\sum_{n=0}^\infty{|a_n\beta(n)|^2_k}\; \text{is convergent}\right\}.$$
 where $\forall n\in \mathbb N, a_n \in \mathbb B \mathbb C$.\\ 
Setting $a_n=a_{n1}e_1+a_{n2}e_2$ and $\beta(n)=\beta_1(n)e_1+\beta_2(n)e_2$ one gets:
$$|a_n\beta(n)|^2_k=e_1|a_{n1}\beta_1(n)|^2+e_2|a_{n2}\beta_2(n)|^2$$
which means that both complex series $$\sum_{n=0}^\infty{|a_{n1}\beta_1(n)|^2},\;\; \sum_{n=0}^\infty{|a_{n2}\beta_2(n)|^2}$$ are convergent and thus both the functions 
$$f_1(z_1)=\sum_{n=0}^\infty{ a_{n1}z^n_1},\;\;\; f_2(z_2)=\sum_{n=0}^\infty{a_{n2}z^n_2}$$ belong to the weighted Hardy space of the unit disk $H^2(\mathbb{D}).$
This means that bicomplex weighted Hardy space can be written as
$$H^2_\beta(\mathbb{D}_{\mathbb{B}\mathbb{C}})=e_1H^2_{\beta_1}(\mathbb{D})+e_2H^2_{\beta_2}(\mathbb{D})$$
The $\mathbb{B}\mathbb{C}$-valued inner product on $H^2_\beta(\mathbb{D}_{\mathbb{B}\mathbb{C}})$ is given by 
$$<f,g>_{H^2_\beta(\mathbb{D}_{\mathbb{B}\mathbb{C}})} = \sum_{n=0}^\infty{a_n\beta(n)\;.\;(b_n\beta(n))}^{\dagger_3}$$
and it generates the $\mathbb{D}$-valued norm:
\begin{eqnarray*}
\|f\|^2_{\mathbb{D},{H^2_\beta(\mathbb{D}_{\mathbb{B}\mathbb{C}})}}&=&<f,f>_{H^2_\beta(\mathbb{D}_{\mathbb{B}\mathbb{C}})}=\sum_{n=0}^\infty{|a_n\beta(n)|^2_k}\\
&=&e_1\sum_{n=0}^\infty{|a_{n1}\beta_1(n)|^2}+e_2\sum_{n=0}^\infty{|a_{n2}\beta_2(n)|^2}\\ 
&=&e_1\|a_{n1}\beta_1(n)\|^2_{H^2_{\beta_1}(\mathbb{D})}+e_2\|a_{n2}\beta_2(n)\|^2_{H^2_{\beta_2}(\mathbb{D})} 
 \end{eqnarray*}
\end{definition}

\begin{remark}\;The mapping $T:H^2_\beta(\mathbb{D}_{\mathbb{B}\mathbb{C}})\rightarrow l^2_\beta({\mathbb{B}\mathbb{C}})$ defined as 
$$ T(f(Z))=T(\sum_{n=0}^\infty{a_nZ^n})=\left\{a_n\right\}_{n=0}^\infty$$ is an isomorphism.
Moreover, $T$ also preserves norm.
\end{remark}

\begin{definition} Let $\Pi^+$ denotes the usual complex upper half plane. Define $\Pi^+(\mathbb{B}\mathbb{C})=\left\{Z=z+jw=e_1z_1+e_2z_2\;|\;(z_1,z_2)\in (\Pi^+\times \Pi^+)\right\}$. Then we say that $\Pi^+(\mathbb{B}\mathbb{C})$ is a bicomplex upper half plane.  
From the above definition, $Z\in \Pi^+(\mathbb{B}\mathbb{C})$ if and only if $\text{Img}(z_1\geq 0),\;\text{Img}(z_2\geq 0)$,
so that one can write $$\Pi^+(\mathbb{B}\mathbb{C})=e_1\Pi^++e_2\Pi^+.$$
\end{definition}
For $1\leq p< \infty $, we define bicomplex Hardy space over bicomplex upper  half plane  as
$$H^p(\Pi^+(\mathbb{B}\mathbb{C}))=e_1 H^p_1(\Pi^+)+e_2 H^p_2(\Pi^+),$$ where $ H^p_1(\Pi^+)$ and $H^p_2(\Pi^+)$ denotes the usual Hardy spaces of the complex upper half plane.\\
For $p=\infty$, we define  $H^\infty(\Pi^+(\mathbb{B}\mathbb{C}))$ to be the space of all bounded holomorphic functions such that 
  $$ H^\infty(\Pi^+(\mathbb{B}\mathbb{C}))=e_1 H^\infty_1(\Pi^+_1)+e_2 H^\infty_2(\Pi^+_2),$$ where $H^\infty_1(\Pi^+_1),H^\infty_2(\Pi^+_2)$ have their usual meanings in complex plane.
\begin{example} Define a linear transformation $$L(W)=i\frac{1+W}{1-W}\;.$$ Then we can write 
$$L(W)= i\left(\frac{1+w_1}{1-w_1}\right)e_1+ i\left(\frac{1+w_2}{1-w_2}\right)e_2,$$ so that $L=e_1L_1+e_2L_2$, where $L_1,L_2$ are the linear transformations mapping the unit disc to the upper half plane in $\mathbb{C}$. Moreover, $L$ maps $\mathbb{D}_{\mathbb{B}\mathbb{C}}$ onto $\Pi^+$.\\
Further $L$ is invertible if and only if $L_1$ and $L_2$ are invertible, i.e.,$$L^{-1}=e_1L^{-1}_1+e_2L^{-1}_2.$$
\end{example}  
\begin{remark} Any bicomplex Holomorphic Function space $F(\Omega)$, where $\Omega$ is any open set in $\mathbb{B}\mathbb{C}$ can be decomposed as $e_1F(\Omega_1)+e_2F(\Omega_2)$, where $F(\Omega_i)$ for $i=1,2$ are usual classical function spaces and $\Omega_i$ are the open sets in $\mathbb{C}$.\\
Moreover, the bicomplex function space $F(\Omega)$ will be a Hilbert space, Banach space and Frechet space provided the corresponding decomposed copies of function spaces $F(\Omega_i)$ for $i=1,2$ are Hilbert spaces, Banach spaces and Frechet spaces respectively.
\end{remark}  
  
\end{section}  
\begin{section}  {Composition Operators on $H^2(\mathbb{D}_{\mathbb{B}\mathbb{C}})$}

Let $f:\mathbb{D}_{\mathbb{B}\mathbb{C}}\rightarrow \mathbb{B}\mathbb{C}$ be a bicomplex holomorphic function and $\Phi:\mathbb{D}_{\mathbb{B}\mathbb{C}}\rightarrow \mathbb{D}_{\mathbb{B}\mathbb{C}}$ be a bicomplex holomorphic function on $\mathbb{D}_{\mathbb{B}\mathbb{C}}$. Then we can decompose $$f(Z)=f(z_1+jz_2)=f_1(z_1-iz_2)e_1+f_2(z_1+iz_2)e_2,$$ where $f_i$ for $i=1,2$ are usual complex-valued holomorphic functions on unit disc $\mathbb{D}$.\\
Also $\Phi(Z)=\Phi(z_1+jz_2)=e_1\Phi_1(z_1-iz_2)+e_2\Phi_(z_1+iz_2)$, where $\Phi_i:\mathbb{D}\rightarrow \mathbb{D}$ for $i=1,2$ are usual complex-valued holomorphic functions.
Then by an easy calculation we have the following lemma:
\begin{lemma}\label{2} Let Let $f:\mathbb{D}_{\mathbb{B}\mathbb{C}}\rightarrow \mathbb{B}\mathbb{C}$ be a bicomplex holomorphic function and $\Phi:\mathbb{D}_{\mathbb{B}\mathbb{C}}\rightarrow \mathbb{D}_{\mathbb{B}\mathbb{C}}$ be a bicomplex holomorphic function on $\mathbb{D}_{\mathbb{B}\mathbb{C}}$. 
Then $f\;o\;\Phi:\mathbb{D}_{\mathbb{B}\mathbb{C}}\rightarrow \mathbb{B}\mathbb{C}$ defined by
$$(f\;o\;\Phi)(z_1+jz_2)=e_1(f_1\;o\;\Phi_1)(z_1-iz_2)+e_2(f_2\;o\;\Phi_2)(z_1+iz_2)$$ is a bicomplex holomorphic function if and only if $f_1\;o\;\Phi_1:\mathbb{D}\rightarrow\mathbb{C}$ and $f_2\;o\;\Phi_2:\mathbb{D}\rightarrow\mathbb{C}$ are complex-valued holomorphic functions.
\end{lemma}

 \begin{theorem} $(Littlewoods Subordination Principle)$\\
Suppose $\Phi$ is a bicomplex holomorphic self map of $\mathbb{D}_{\mathbb{B}\mathbb{C}}$, with $\Phi(0)=0$. Then for $f\in H^2(\mathbb{D}_{\mathbb{B}\mathbb{C}})$,
 $$C_\Phi f\in H^2(\mathbb{D}_{\mathbb{B}\mathbb{C}})\;\;\;\; \text{and}\;\;\; \;\|C_\Phi f\|_{\mathbb D,H^2(\mathbb{D}_{\mathbb{B}\mathbb{C}})}\leq \|f\|_{\mathbb D,H^2(\mathbb{D}_{\mathbb{B}\mathbb{C}})},$$ where $C_\Phi f$ is a composition operator defined as $C_\Phi f=fo\Phi$. 
\end{theorem}
\begin{proof} Using Lemma (\ref{2}) and the Closed graph theorem in {\cite{LL}}, it is clear that $C_\Phi f\in H^2(\mathbb{D}_{\mathbb{B}\mathbb{C}})$. Further by using the classical Littlewood's Principle, we have 
\begin{eqnarray*} \|C_\Phi f\|_{\mathbb D,H^2(\mathbb D_{\mathbb B \mathbb C})}&=&\|f\;o\;\Phi\|_{D,{H^2(\mathbb D_{\mathbb B \mathbb C})}}\\
&=& <f\;o\;\Phi,f\;o\;\Phi>^\frac{1}{2}_{H^2(\mathbb D_{\mathbb B \mathbb C})}\\
&=& e_1<f_1\;o\;\Phi_1,f_1\;o\;\Phi_1>^\frac{1}{2}_{H^2(\mathbb D)}+e_2<f_2\;o\;\Phi_2,f_2\;o\;\Phi_2>^\frac{1}{2}_{H^2(\mathbb D)}\\
&=& e_1\|f_1\;o\;\Phi_1\|_{H^2(\mathbb{D})}+e_2\|f_2\;o\;\Phi_2\|_{H^2(\mathbb{D})}\\  
&\leq & e_1\|f_1\|_{H^2(\mathbb{D})}+e_2\|f_2\|_{H^2(\mathbb{D})}\\
&=& \|f\|_{\mathbb D,H^2(\mathbb{D}_{\mathbb{B}\mathbb{C}})}.
\end{eqnarray*}
\end{proof}
\begin{example} For each point $a \in \mathbb{D}_{\mathbb{B}\mathbb{C}}$, let $T_a:\mathbb{D}_{\mathbb{B}\mathbb{C}} \rightarrow \mathbb{D}_{\mathbb{B}\mathbb{C}}$ be a linear transformation defined by
$$T_a(Z)=\frac{a-Z}{1-a^{\dagger_3}Z}=e_1\frac{a_1-z_1}{1-\overline a_1z_1}+e_2\frac{a_2-z_2}{1-\overline a_2z_2}$$
$$= e_1T_{a_1}+e_2T_{a_2}$$
where $T_{a_1}\;\text{and}\; T_{a_2}$ are usual linear transformations maps $\mathbb{D}$ onto itself.\\
Then $T_a$ has following properties:\\ 
(1) $T_a$ maps $\mathbb{D}_{\mathbb{B}\mathbb{C}}$ onto itself.\\ 		 
(2) $T_a(0)=a,\;\; T_a(a)=0,\;\; T_a\;o\;T_a(Z)=Z,\; \forall\; Z\in \mathbb{D}_{\mathbb{B}\mathbb{C}}.$\\
(3) $T_a$ is invertible if and only if $T_{a1}$ and $T_{a2}$ are invertible.\\\\
\end{example}
\begin{lemma} For each $a\in \mathbb{D}_{\mathbb{B}\mathbb{C}}$, the operator $C_{T_a}$ is bounded on $H^2(\mathbb{D}_{\mathbb{B}\mathbb{C}})$. Moreover,           $$\|C_{T_a}\|_{\mathbb D,H^2(\mathbb{D}_{\mathbb{B}\mathbb{C}})}\leq \left(\frac{1+|a|_k}{1-|a|_k}\right)^{\frac{1}{2}}.$$
\end{lemma}
\begin{proof} Using Lemma (5.1), we can write\\
  $\|f\;o\;T_a\|^2_{\mathbb D,H^2(\mathbb D_{\mathbb{B}\mathbb{C}})}= e_1\|f_1\;o\;T_{a_1}\|^2_{H^2(\mathbb{D})}+e_2\|f_2\;o\;T_{a_2}\|^2_{H^2(\mathbb{D})}.$\\
Since $f_i\;o\;T_{a_i} \in H^2(\mathbb{D})$, for $i=1,2$, so by using Lemma (see, \cite{SS} at P. 16\;), we get $C_{T_{a_i}}$ is bounded on $H^2(\mathbb{D})$ and   
$$\|C_{T_{a_i}}fi\|^2_{H^2(\mathbb{D})} \leq \left(\frac{1+|a_i|}{1-|a_i|}\right)\|f_i\|^2_{H^2(\mathbb{D})}$$ for $i=1,2.$ 
Thus, $C_{T_a}$ is bounded on $H^2(\mathbb{D}_{\mathbb{B}\mathbb{C}})$.
\begin{eqnarray*} \text{ Further,}\;\;\;\;  
   \|C_{T_a}f\|^2_{\mathbb D,H^2(\mathbb{D}_{\mathbb{B}\mathbb{C}})}&=& e_1\|C_{T_{a_1}}f_1\|^2_{H^2(\mathbb{D})}+e_2\|C_{T_{a_2}}f_2\|^2_{H^2(\mathbb{D})}\\
   &\leq& e_1\left(\frac{1+|a_1|}{1-|a_1|}\right)\|f_1\|^2_{H^2(\mathbb{D})}+e_2\left(\frac{1+|a_2|}{1-|a_2|}\right)\|f_2\|^2_{H^2(\mathbb{D})}\\
&=& \left(\frac{1+|a|_k}{1-|a|_k}\right)\|f\|^2_{H^2(\mathbb{D}_{\mathbb{B}\mathbb{C}})}.
\end{eqnarray*}
\end{proof}
 \begin{theorem} $(Littlewood's Theorem)$:
Suppose $\Phi$ is a holomorphic self-map of $\mathbb{D}_{\mathbb{B}\mathbb{C}}$. Then $C_\Phi$ is a bounded operator on $H^2(\mathbb{D}_{\mathbb{B}\mathbb{C}})$ and $$ \|C_\Phi\|_{\mathbb D,H^2(\mathbb{D}_{\mathbb{B}\mathbb{C}})} \leq \sqrt{\frac{1+|\Phi(0)|_k}{1-|\Phi(0)|_k}}.$$ 
\end{theorem}
\begin{proof} Suppose $\Phi(0)=a$ and write $\Psi=T_a \;o\; \Phi$. Then $\Psi$ is a bicomplex holomorphic function takes $\mathbb{D}_{\mathbb{B}\mathbb{C}}$ onto itself and fixes the origin. Also by the self-inverse property of $T_a$ we have $\Phi=T_a\;o\;\Psi$. Let $C_\Phi=C_\Psi C_{T_a}$, then by using the Lemma (5.4) and Littlewood's Subordination Principle, it is clear that $C_\Phi$ is a bounded operator on $H^2(\mathbb{D}_{\mathbb{B}\mathbb{C}})$. Moreover, 
$$ \|C_\Phi\|_{\mathbb D,H^2(\mathbb{D}_{\mathbb{B}\mathbb{C}})} \leq \sqrt{\frac{1+|\Phi(0)|_k}{1-|\Phi(0)|_k}}.$$ 
\end{proof}
\end{section}

\bibliographystyle{amsplain}

\noindent Romesh Kumar, \textit{Department of Mathematics, University of Jammu, Jammu, J\&K - 180 006, India.}\\
E-mail :\textit{ romesh\_jammu@yahoo.com}\\

\noindent Kulbir Singh, \textit{Department of Mathematics, University of Jammu, Jammu,  J\&K - 180 006, India.}\\
E-mail :\textit{ singhkulbir03@yahoo.com}\\

\end{document}